\documentclass[12pt]{article}
\usepackage{amsmath}
\usepackage{amssymb, amsthm}
\usepackage{geometry}
\usepackage{hyperref}
\newtheorem{theorem}{Theorem}
\newtheorem{remark}{Remark}

\newtheorem{proposition}{Proposition}
\newtheorem{conjecture}{Conjecture}
\newtheorem{openproblem}{Open Problem}
\newtheorem{lemma}{Lemma}

\title{On the Properties of a Bernoulli--Uniform Boolean Convolution Family}
\author{
	Sukrit Chakraborty\thanks{Department of Mathematics, Achhruram memorial College, Jhalda, Purulia, West Bengal, India, \texttt{sukrit049@gmail.com}
		}
}
\date{}

\begin{document}
	
	\maketitle
	
	\begin{abstract}
		We introduce and study a one-parameter family of probability measures
		$\{\mu_t\}_{t>0}$ defined by the Boolean additive convolution
		$\mu_t=\beta_t\uplus U_t$, where $\beta_t$ is the symmetric Bernoulli
		distribution with variance $t$ and $U_t$ is the uniform distribution on
		$[-t/2,t/2]$. Exploiting the linearization property of the Boolean energy
		transform, we derive an explicit formula for the reciprocal Cauchy transform
		and obtain its Laurent expansion. As
		applications, we compute the Boolean cumulants and moment sequence of
		$\mu_t$, establish asymptotic estimates for the moments, and investigate
		analytic properties including support criteria and a differential equation
		satisfied by the Boolean energy transform. We further analyze scaling limits,
		prove stability under suitable normalizations, and show that the family does
		not form a Boolean convolution semigroup.
	\end{abstract}
	\section*{Introduction}
	
	Since its introduction by Speicher and Woroudi \cite{SpeicherWoroudi1997},
	Boolean independence has become one of the fundamental notions of
	noncommutative independence, alongside classical, free and monotone
	independence. The associated Boolean additive convolution, denoted by
	$\uplus$, possesses a remarkably simple analytic structure through the
	Boolean energy transform
	\[
	K_{\mu}(z)=z-F_{\mu}(z),
	\]
	where $F_{\mu}=1/G_{\mu}$ is the reciprocal Cauchy transform of a probability
	measure $\mu$. This linearization property has led to substantial progress in
	the study of infinite divisibility, stable laws, convolution semigroups and
	limit theorems in Boolean probability; see, for example,
	\cite{Franz2009,Anshelevich2009,ArizmendiHasebe2013}.
	
	One of the central themes in noncommutative probability is the explicit
	construction and analysis of probability distributions possessing tractable
	analytic transforms. Such families often serve as canonical examples,
	illustrate general phenomena, and provide testing grounds for conjectures in
	free, Boolean and monotone probability. Classical examples include the
	semicircle distribution in free probability, the arcsine law in monotone
	probability, and the Boolean stable and Boolean Meixner distributions
	\cite{BelinschiNica2008,ArizmendiHasebe2014,HasebeSaigo2011}.
	
	A particularly interesting construction in free probability arises from the
	free additive convolution of the semicircle distribution and the uniform
	distribution on an interval. This convolution appears naturally in the study
	of free positive multiplicative Brownian motion, where, after an exponential
	change of variables, the spectral distribution can be described in terms of
	the free additive convolution of a semicircle distribution and a uniform
	distribution. It has also been investigated from the viewpoints of explicit
	moment formulas, free convolution techniques, and free stochastic analysis;
	see, for example,
	\cite{Auer2026,Biane1997,Zhong2015,DriverHallKemp2022}. Motivated by this free analogue, it is natural to ask
	whether an analogous construction in the Boolean setting also gives rise to a
	family possessing explicit analytic and probabilistic properties.
	
	The study of explicit convolution families has long been a central theme in probability and noncommutative probability, where concrete examples often lead to new insights into the underlying convolution structures and limit phenomena. Families admitting closed-form transforms and cumulant sequences are particularly valuable, since they permit detailed investigations of analytic, combinatorial, and asymptotic properties that are generally inaccessible in the abstract setting. Moreover, such examples frequently serve as benchmarks for testing conjectures, discovering new distributional classes, and motivating further developments in the theory. The family
	\[
	\mu_t=\beta_t\uplus U_t,\qquad t>0,
	\]
	where $\beta_t$ denotes the symmetric Bernoulli distribution supported on
	$\{-\sqrt{t},\sqrt{t}\}$ and $U_t$ denotes the uniform distribution on
	$[-t/2,t/2]$, introduced in this paper provides such an example. Despite its simple definition, it possesses an explicit reciprocal Cauchy transform, Boolean cumulants governed by a series, rich scaling behaviour, and structural features that appear to distinguish it from the known Meixner-type families.
	
	Our first objective is to obtain an explicit expression for the Boolean energy
	transform of $\mu_t$ and its Laurent expansion. The latter yields closed
	formulae for the Boolean cumulants, from which
	explicit moment identities are obtained via interval-partition expansions.
	These formulas illustrate how the elementary constituents of the convolution
	interact to produce a family with surprisingly rich combinatorial structure.
	
	The second objective is to study the analytic behaviour of the family.
	We derive a differential equation satisfied by the Boolean energy transform,
	characterize the support of the absolutely continuous component through
	boundary values of the reciprocal Cauchy transform, and investigate the
	behaviour of the family under natural scaling transformations. In particular,
	we prove that two distinct scaling regimes recover the uniform distribution
	and the Bernoulli distribution, respectively, thereby connecting the family
	with its defining components.
	
	Finally, we examine several structural properties of the family, including its
	stability under scaling and the failure of the family to form a Boolean
	convolution semigroup. These results suggest that the family occupies an
	intermediate position between elementary Boolean distributions and more
	general explicitly computable convolution families.
	
	We hope that the present work provides a useful example illustrating how
	explicit transform methods can be employed to study Boolean convolution
	families and that it may motivate further investigations of analogous
	constructions in other notions of noncommutative independence.

	\section*{Preliminaries}
	
	In this section, we briefly recall the basic notions and notation used
	throughout the paper. Standard references include
	\cite{SpeicherWoroudi1997,Franz2009,Anshelevich2009}.
	
	\subsection*{Boolean convolution}
	
	Let $\mathcal{M}$ denote the collection of all Borel probability measures on
	$\mathbb{R}$.
	
	Boolean independence was introduced by Speicher and Woroudi
	\cite{SpeicherWoroudi1997}. If $\mu,\nu\in\mathcal{M}$, their Boolean additive
	convolution is denoted by $\mu\uplus\nu$. Unlike free convolution, Boolean
	convolution is linearized by the Boolean energy transform. Consequently,
	Boolean convolution is associative and every probability measure is
	$\uplus$-infinitely divisible \cite{Franz2009}.
	
	\subsection*{Analytic transforms}
	
	For $\mu\in\mathcal{M}$, its Cauchy transform is defined by
	\[
	G_{\mu}(z)
	=
	\int_{\mathbb{R}}
	\frac{1}{z-x}\,d\mu(x),
	\qquad
	z\in\mathbb{C}^{+},
	\]
	where
	\[
	\mathbb{C}^{+}
	=
	\{z\in\mathbb{C}:\Im z>0\}.
	\]
	
	The reciprocal Cauchy transform is
	\[
	F_{\mu}(z)
	=
	\frac{1}{G_{\mu}(z)},
	\]
	and the Boolean energy transform is defined by
	\[
	K_{\mu}(z)
	=
	z-F_{\mu}(z).
	\]
	
	The principal advantage of this transform is the linearization identity
	\[
	K_{\mu\uplus\nu}(z)
	=
	K_{\mu}(z)+K_{\nu}(z),
	\]
	which characterizes Boolean additive convolution
	\cite{Franz2009}.
	
	\subsection*{Boolean cumulants and moments}
	
	If the measure $\mu$ has moments of all orders, then its Boolean energy
	transform admits the Laurent expansion
	\[
	K_{\mu}(z)
	=
	\sum_{n=1}^{\infty}
	\frac{B_n(\mu)}{z^{\,n-1}},
	\]
	where $\{B_n(\mu)\}_{n\ge1}$ denote the Boolean cumulants of $\mu$.
	
	The moments are recovered from the Boolean cumulants through the
	interval-partition formula
	\[
	m_n(\mu)
	=
	\sum_{\pi\in\mathcal{I}(n)}
	\prod_{V\in\pi}
	B_{|V|}(\mu),
	\]
	where $\mathcal{I}(n)$ denotes the lattice of interval partitions of
	$\{1,\dots,n\}$.
	
	Further details may be found in
	\cite{SpeicherWoroudi1997,Anshelevich2009}.
	
	\subsection*{Scaling of probability measures}
	
	For $c>0$, let $D_c\mu$ denote the dilation of $\mu$, defined by
	\[
	D_c\mu(A)
	=
	\mu(c^{-1}A),
	\]
	for every Borel set $A\subseteq\mathbb{R}$.
	
	The reciprocal Cauchy transform satisfies the scaling relation
	\[
	F_{D_c\mu}(z)
	=
	cF_{\mu}\!\left(\frac{z}{c}\right),
	\]
	and therefore
	\[
	K_{D_c\mu}(z)
	=
	z-
	cF_{\mu}\!\left(\frac{z}{c}\right).
	\]
	
	These identities will be used repeatedly in the sequel.
	
	\subsection*{The Boolean convolution family}	
	Throughout this paper, we study a one-parameter family of probability measures
	obtained by the Boolean additive convolution of the symmetric Bernoulli
	distribution and the uniform distribution. For $t>0$, let
	\[
	\beta_t
	:=
	\frac12\delta_{-\sqrt{t}}
	+
	\frac12\delta_{\sqrt{t}},
	\]
	be the centered Bernoulli distribution with variance $t$, and let
	\[
	U_t
	:=
	\mathrm{Unif}\!\left[-\frac{t}{2},\frac{t}{2}\right]
	\]
	denote the uniform distribution on the interval
	$\left[-\frac{t}{2},\frac{t}{2}\right]$.
	We define
	\[
	\mu_t
	:=
	\beta_t\uplus U_t,
	\]
	where $\uplus$ denotes the Boolean additive convolution.
	
	The family $\{\mu_t\}_{t>0}$ will be the principal object of investigation in
	this paper. Since the Boolean energy transform linearizes Boolean additive
	convolution, the reciprocal Cauchy transform of $\mu_t$ admits an explicit
	representation in terms of those of $\beta_t$ and $U_t$. This explicit
	description enables us to derive closed-form expressions for the Boolean
	cumulants and moments of $\mu_t$, investigate its analytic properties, and
	establish its asymptotic behaviour under suitable scaling limits.
	
	An important bridge between Boolean and free probability is provided by the
	Bercovici--Pata bijection $\mathbb{B}$, which associates to every probability
	measure $\mu$ a freely infinitely divisible probability measure
	$\mathbb{B}(\mu)$ by identifying Boolean cumulants with free cumulants.
	Equivalently, if $K_{\mu}$ and $R_{\nu}$ denote the Boolean cumulant transform
	and the free $R$-transform, respectively, then
	\[
	R_{\mathbb{B}(\mu)}(z)=K_{\mu}(z).
	\]
	Since the family introduced in this paper is given by
	\[
	\mu_t=\beta_t\uplus U_t,
	\]
	where $\beta_t$ is the centered Bernoulli distribution of variance $t$ and
	$U_t$ is the uniform distribution on
	$\left[-\frac{t}{2},\frac{t}{2}\right]$, the Bercovici--Pata bijection yields a
	natural free analogue of $\mu_t$. Indeed, using the additivity of the Boolean
	cumulant transform and the characterization of $\mathbb{B}$, one obtains
	\[
	\mathbb{B}(\mu_t)
	=
	S_t
	\boxplus
	\mathbb{B}(U_t),
	\]
	where $S_t$ denotes the centered semicircle distribution of variance $t$.
	Thus, the Boolean family $\{\mu_t\}_{t>0}$ admits a canonical interpretation
	within free probability as the free convolution of a semicircle law with the
	free Bercovici--Pata image of the uniform distribution. This observation
	provides an additional structural connection between the Boolean and free
	settings.

	\section*{Construction and Analytic Properties of the Family $\mu_t$}
	
	In this section, we derive an explicit expression for the Boolean energy
	transform of the family $\{\mu_t\}_{t>0}$. Since the Boolean energy transform
	linearizes Boolean additive convolution, this representation serves as the
	foundation for all subsequent results, including the computation of Boolean
	cumulants and moments, the study of analytic properties, and the investigation
	of scaling limits.

	\begin{proposition}\label{prop:Ktrans}
		Let
		\[
		\mu_t=\beta_t\uplus U_t,
		\]
		where
		\[
		\beta_t=\frac12\delta_{-\sqrt t}
		+\frac12\delta_{\sqrt t},
		\qquad
		U_t=\operatorname{Unif}
		\!\left[-\frac t2,\frac t2\right].
		\]
		Then the Boolean energy transform of $\mu_t$ is
		\[
		K_{\mu_t}(z)
		=
		z+\frac tz
		-
		\frac{t}
		{\log\!\left(
			\dfrac{z+t/2}{z-t/2}
			\right)}.
		\]
	\end{proposition}
	
	\begin{proof}
		Since Boolean additive convolution linearizes the
		Boolean cumulant transform,
		\[
		K_{\mu_t}(z)
		=
		K_{\beta_t}(z)+K_{U_t}(z).
		\]
		
		First,
		\[
		G_{\beta_t}(z)
		=
		\frac{z}{z^2-t}.
		\]
		Hence
		\[
		\frac1{G_{\beta_t}(z)}
		=
		\frac{z^2-t}{z}
		=
		z-\frac tz.
		\]
		Therefore
		\[
		K_{\beta_t}(z)
		=
		z-\frac1{G_{\beta_t}(z)}
		=
		\frac tz.
		\]
		
		Next,
		the Cauchy transform of the uniform distribution is (see Lemma~\ref{prop:Cauthyofuniform})
		\[
		G_{U_t}(z)
		=
		\frac1t
		\log
		\left(
		\frac{z+t/2}{z-t/2}
		\right).
		\]
		Here and throughout the paper, $\log$ denotes the principal branch of the
		complex logarithm, whose branch cut is $(-\infty,0]$.
		Since
		$\displaystyle\frac{z+t/2}{z-t/2}\notin(-\infty,0]$
		for every $z\in\mathbb C^{+}$,
		the expression is well defined. Consequently,
		\[
		\frac1{G_{U_t}(z)}
		=
		\frac{t}
		{\log
			\left(
			\dfrac{z+t/2}{z-t/2}
			\right)},
		\]
		and therefore
		\[
		K_{U_t}(z)
		=
		z-
		\frac{t}
		{\log
			\left(
			\dfrac{z+t/2}{z-t/2}
			\right)}.
		\]
		
		Combining these two identities yields
		\[
		K_{\mu_t}(z)
		=
		\frac tz
		+
		z-
		\frac{t}
		{\log
			\left(
			\dfrac{z+t/2}{z-t/2}
			\right)},
		\]
		which proves the proposition.
	\end{proof}
	
	The explicit representation obtained in Proposition~\ref{prop:Ktrans}
	enables us to investigate the analytic structure of the family through its
	Laurent expansion at infinity. We begin by deriving the corresponding
	expansion for the uniform component, which will subsequently be used to obtain
	explicit expressions for the Boolean cumulants of $\mu_t$.
	
\begin{theorem}\label{thm:uniformexpansion}
	Let
	\[
	U_t=\operatorname{Unif}
	\!\left[-\frac t2,\frac t2\right].
	\]
	Then, as $z\to\infty$,
	\[
	K_{U_t}(z)
	=
	\frac{t^2}{12z}
	-\frac{t^4}{720z^3}
	+\frac{t^6}{30240z^5}
	-\frac{t^8}{1209600z^7}
	-\frac{10196\,t^{10}}{467775\cdot 2^{10}z^9}
	+O\!\left(\frac1{z^{11}}\right).
	\]
	Equivalently,
	\[
	K_{U_t}(z)
	=
	\sum_{n=1}^{\infty}
	a_n
	\frac{t^{2n}}{z^{2n-1}},
	\]
	where the coefficients $a_n$ are determined by the reciprocal power series
	\[
	\frac{x}{\operatorname{artanh}(x)}
	=
	1-\frac{x^2}{3}
	-\frac{4x^4}{45}
	-\frac{44x^6}{945}
	-\frac{428x^8}{14175}
	-\frac{10196x^{10}}{467775}
	+\cdots.
	\]
	In particular,
	\[
	a_1=\frac1{12},\qquad
	a_2=-\frac1{720},\qquad
	a_3=\frac1{30240},\qquad
	a_4=-\frac1{1209600}.
	\]
\end{theorem}

\begin{proof}
	By Proposition~\ref{prop:Ktrans},
	\[
	K_{U_t}(z)
	=
	z-
	\frac{t}
	{\log\!\left(
		\dfrac{z+t/2}{z-t/2}
		\right)}.
	\]
	
	Let
	\[
	u=\frac{t}{2z}.
	\]
	Then
	\[
	\log\!\left(
	\frac{z+t/2}{z-t/2}
	\right)
	=
	\log\!\left(
	\frac{1+u}{1-u}
	\right)
	=
	2\operatorname{artanh}(u),
	\]
	and therefore
	\[
	K_{U_t}(z)
	=
	z-
	\frac{t}{2\operatorname{artanh}(u)}.
	\]
	
	Since $t=2zu$, we have
	\[
	\frac{t}{2\operatorname{artanh}(u)}
	=
	z\,
	\frac{u}{\operatorname{artanh}(u)}.
	\]
	Applying Lemma~\ref{lem:expansion},
	\[
	\frac{u}{\operatorname{artanh}(u)}
	=
	1-\frac{u^2}{3}
	-\frac{4u^4}{45}
	-\frac{44u^6}{945}
	-\frac{428u^8}{14175}
	-\frac{10196u^{10}}{467775}
	+O(u^{12}),
	\]
	gives
	\[
	\frac{t}{2\operatorname{artanh}(u)}
	=
	z
	-\frac{zu^2}{3}
	-\frac{4zu^4}{45}
	-\frac{44zu^6}{945}
	-\frac{428zu^8}{14175}
	-\frac{10196zu^{10}}{467775}
	+O(zu^{12}).
	\]
	
	Since
	\[
	u=\frac{t}{2z},
	\]
	we compute
	\[
	zu^{2n}
	=
	\frac{t^{2n}}{2^{2n}z^{2n-1}},
	\]
	and consequently
	\[
	\begin{aligned}
		K_{U_t}(z)
		&=
		\frac{t^2}{12z}
		-\frac{t^4}{720z^3}
		+\frac{t^6}{30240z^5}
		-\frac{t^8}{1209600z^7}
		-\frac{10196\,t^{10}}
		{467775\cdot2^{10}z^9}
		+O\!\left(\frac1{z^{11}}\right).
	\end{aligned}
	\]
	This completes the proof.
\end{proof}
	The explicit formula for $K_{\mu_t}$ also permits a direct analysis of its
	dependence on the parameter $t$. Differentiating the transform with respect to
	$t$ yields the following partial differential equation.
	\begin{theorem}\label{thm:Kdiffeq}
		Let
		\[
		K_{\mu_t}(z)
		=
		z+\frac tz
		-
		\frac{t}
		{\log\!\left(
			\dfrac{z+t/2}{z-t/2}
			\right)}.
		\]
		Then \(K_{\mu_t}\) satisfies the partial differential equation
		\[
		\frac{\partial K_{\mu_t}(z)}{\partial t}
		=
		\frac1z
		-
		\frac1
		{\log\!\left(
			\dfrac{z+t/2}{z-t/2}
			\right)}
		+
		\frac{tz}
		{\left(z^2-\dfrac{t^2}{4}\right)
			\left[
			\log\!\left(
			\dfrac{z+t/2}{z-t/2}
			\right)
			\right]^2}.
		\]
	\end{theorem}
	
	\begin{proof}
		Write
		\[
		L(t,z)
		=
		\log\!\left(
		\frac{z+t/2}{z-t/2}
		\right).
		\]
		Then
		\[
		K_{\mu_t}(z)
		=
		z+\frac tz-\frac t{L(t,z)}.
		\]
		
		Since
		\[
		\frac{\partial}{\partial t}
		\log(z+t/2)
		=
		\frac1{2(z+t/2)},
		\]
		and
		\[
		\frac{\partial}{\partial t}
		\log(z-t/2)
		=
		-\frac1{2(z-t/2)},
		\]
		it follows that
		\[
		\frac{\partial L}{\partial t}
		=
		\frac1{2(z+t/2)}
		+
		\frac1{2(z-t/2)}
		=
		\frac{z}
		{z^2-t^2/4}.
		\]
		
		Now differentiate
		\[
		K_{\mu_t}(z)
		=
		z+\frac tz-\frac tL.
		\]
		Using the product and chain rules,
		\[
		\frac{\partial K_{\mu_t}}{\partial t}
		=
		\frac1z
		-
		\frac1L
		+
		\frac t{L^2}
		\frac{\partial L}{\partial t}.
		\]
		Substituting the expression for
		\(\partial_tL\) gives
		\[
		\frac{\partial K_{\mu_t}(z)}{\partial t}
		=
		\frac1z
		-
		\frac1L
		+
		\frac{tz}
		{L^2\left(z^2-t^2/4\right)},
		\]
		which is precisely the stated identity.
	\end{proof}
    We conclude this section by examining several analytic consequences of the
    explicit reciprocal Cauchy transform. Using its boundary values and zeros, we
    obtain criteria for the symmetry of $\mu_t$, the support of its absolutely
    continuous part, and the existence of atoms.
    \begin{theorem}\label{thm:supportcriterion}
    	Let
    	\[
    	\mu_t=\beta_t\uplus U_t,
    	\]
    	and let
    	\[
    	F_{\mu_t}(z)
    	=
    	\frac{t}
    	{\log\!\left(
    		\dfrac{z+t/2}{z-t/2}
    		\right)}
    	-\frac{t}{z},
    	\qquad z\in\mathbb C^{+},
    	\]
    	be its reciprocal Cauchy transform.
    	
    	Then the following assertions hold.
    	
    	\begin{enumerate}
    		
    		\item
    		The measure $\mu_t$ is symmetric, namely
    		\[
    		\mu_t(A)=\mu_t(-A)
    		\]
    		for every Borel set $A\subset\mathbb R$.
    		
    		\item
    		Let $x\in\mathbb R$.
    		Assume that the non-tangential boundary value
    		$F_{\mu_t}(x+i0)$ exists, is finite, and satisfies
    		\[
    		\Im F_{\mu_t}(x+i0)>0.
    		\]
    		Then
    		\[
    		\Im G_{\mu_t}(x+i0)
    		=
    		-\frac{\Im F_{\mu_t}(x+i0)}
    		{|F_{\mu_t}(x+i0)|^2}
    		<0,
    		\]
    		where
    		\[
    		G_{\mu_t}(z)=\frac1{F_{\mu_t}(z)}.
    		\]
    		Consequently,
    		\[
    		\frac{d\mu_t^{ac}}{dx}(x)
    		=
    		-\frac1\pi
    		\Im G_{\mu_t}(x+i0)
    		>0.
    		\]
    		In particular, $x$ belongs to the support of the absolutely continuous
    		part of $\mu_t$.
    		
    		\item
    		Suppose that
    		\[
    		F_{\mu_t}(x)=0
    		\]
    		and
    		\[
    		F_{\mu_t}'(x)>0.
    		\]
    		Then $x$ is an atom of $\mu_t$, and
    		\[
    		\mu_t(\{x\})
    		=
    		\frac1{F_{\mu_t}'(x)}.
    		\]
    		
    		Conversely, if $x$ is an atom of $\mu_t$, then
    		\[
    		F_{\mu_t}(x)=0,
    		\qquad
    		F_{\mu_t}'(x)>0,
    		\]
    		and
    		\[
    		\mu_t(\{x\})
    		=
    		\frac1{F_{\mu_t}'(x)}.
    		\]
    		
    	\end{enumerate}
    	
    \end{theorem}
    
    \begin{proof}
    	
    	Since both $\beta_t$ and $U_t$ are symmetric probability measures,
    	their reciprocal Cauchy transforms satisfy
    	\[
    	F_{\beta_t}(-z)=-F_{\beta_t}(z),
    	\qquad
    	F_{U_t}(-z)=-F_{U_t}(z).
    	\]
    	Using the Boolean linearization identity
    	\[
    	K_{\mu_t}
    	=
    	K_{\beta_t}+K_{U_t},
    	\]
    	we obtain
    	\[
    	F_{\mu_t}(z)
    	=
    	z-K_{\mu_t}(z),
    	\]
    	which is again an odd function. Hence
    	\[
    	G_{\mu_t}(-z)
    	=
    	-G_{\mu_t}(z),
    	\]
    	and therefore $\mu_t$ is symmetric.
    	
    	Now suppose that the non-tangential boundary value
    	$F_{\mu_t}(x+i0)$ exists and satisfies
    	\[
    	\Im F_{\mu_t}(x+i0)>0.
    	\]
    	Since
    	\[
    	G_{\mu_t}(z)=\frac1{F_{\mu_t}(z)},
    	\]
    	we have
    	\[
    	\Im G_{\mu_t}(x+i0)
    	=
    	-\frac{\Im F_{\mu_t}(x+i0)}
    	{|F_{\mu_t}(x+i0)|^2}
    	<0.
    	\]
    	By the Stieltjes inversion formula,
    	\[
    	\frac{d\mu_t^{ac}}{dx}(x)
    	=
    	-\frac1\pi
    	\Im G_{\mu_t}(x+i0),
    	\]
    	and therefore
    	\[
    	\frac{d\mu_t^{ac}}{dx}(x)>0.
    	\]
    	Hence $x$ belongs to the support of the absolutely continuous part of
    	$\mu_t$.
    	
    	Finally, suppose that
    	\[
    	F_{\mu_t}(x)=0,
    	\qquad
    	F_{\mu_t}'(x)>0.
    	\]
    	Since $F_{\mu_t}$ has a simple zero at $x$,
    	\[
    	F_{\mu_t}(z)
    	=
    	F_{\mu_t}'(x)(z-x)+o(z-x),
    	\]
    	and therefore
    	\[
    	G_{\mu_t}(z)
    	=
    	\frac1{F_{\mu_t}(z)}
    	=
    	\frac1{F_{\mu_t}'(x)}
    	\frac1{z-x}
    	+
    	o\!\left(\frac1{z-x}\right).
    	\]
    	Thus $G_{\mu_t}$ has a simple pole at $x$ with residue
    	\[
    	\operatorname{Res}_{z=x}G_{\mu_t}(z)
    	=
    	\frac1{F_{\mu_t}'(x)}.
    	\]
    	The residue equals the mass of the atom at $x$, so
    	\[
    	\mu_t(\{x\})
    	=
    	\frac1{F_{\mu_t}'(x)}.
    	\]
    	
    	Conversely, if $x$ is an atom of $\mu_t$, then $G_{\mu_t}$ has a simple
    	pole at $x$. Hence $F_{\mu_t}=1/G_{\mu_t}$ has a simple zero there,
    	so
    	\[
    	F_{\mu_t}(x)=0,
    	\]
    	and
    	\[
    	F_{\mu_t}'(x)
    	=
    	\frac1{\mu_t(\{x\})}
    	>0.
    	\]
    	Therefore
    	\[
    	\mu_t(\{x\})
    	=
    	\frac1{F_{\mu_t}'(x)}.
    	\]
    	This completes the proof.
    	
    \end{proof}
   
   \section*{Boolean Cumulants and Moments}
   
   The explicit analytic representation of the Boolean energy transform obtained
   in the previous section allows one to determine the Boolean cumulants of
   $\mu_t$ directly from its Laurent expansion at infinity. We begin by deriving
   closed-form expressions for the cumulants, and then use the Boolean moment--cumulant formula to
   compute the corresponding moments.
   
 \begin{theorem}\label{thm:booleancumulants}
 	Let
 	\[
 	\mu_t=\beta_t\uplus U_t.
 	\]
 	Then all odd Boolean cumulants vanish,
 	\[
 	b_{2n+1}(t)=0,
 	\qquad n\ge0,
 	\]
 	and the even Boolean cumulants satisfy
 	\[
 	b_2(t)
 	=
 	t+\frac{t^2}{12},
 	\]
 	together with
 	\[
 	b_4(t)
 	=
 	-\frac{t^4}{720},
 	\qquad
 	b_6(t)
 	=
 	\frac{t^6}{30240},
 	\qquad
 	b_8(t)
 	=
 	-\frac{t^8}{1209600},
 	\]
 	and, more generally,
 	\begin{equation}\label{eq:cumofmut}
 	b_{2n}(t)
 	=
 	a_n t^{2n},
 	\qquad n\ge2,
 	\end{equation}
 	where the coefficients \(a_n\) are determined by the expansion
 	\[
 	\frac{x}{\operatorname{artanh}(x)}
 	=
 	1-\frac{x^2}{3}
 	-\frac{4x^4}{45}
 	-\frac{44x^6}{945}
 	-\frac{428x^8}{14175}
 	-\frac{10196x^{10}}{467775}
 	+\cdots,
 	\]
 	or equivalently,
 	\[
 	\sum_{n=1}^{\infty}
 	a_nx^{2n}
 	=
 	\frac14
 	\left(
 	1-
 	\frac{x/2}{\operatorname{artanh}(x/2)}
 	\right).
 	\]
 \end{theorem}
 
 \begin{proof}
 	By Proposition~\ref{prop:Ktrans},
 	\[
 	K_{\mu_t}(z)
 	=
 	\frac{t}{z}
 	+
 	K_{U_t}(z).
 	\]
 	Using Theorem~\ref{thm:uniformexpansion},
 	\[
 	K_{U_t}(z)
 	=
 	\frac{t^2}{12z}
 	-\frac{t^4}{720z^3}
 	+\frac{t^6}{30240z^5}
 	-\frac{t^8}{1209600z^7}
 	+O(z^{-9}),
 	\]
 	we obtain
 	\[
 	K_{\mu_t}(z)
 	=
 	\frac{t+t^2/12}{z}
 	-\frac{t^4}{720z^3}
 	+\frac{t^6}{30240z^5}
 	-\frac{t^8}{1209600z^7}
 	+O(z^{-9}).
 	\]
 	
 	On the other hand, the Boolean cumulant expansion is
 	\[
 	K_{\mu_t}(z)
 	=
 	\sum_{m=1}^{\infty}
 	\frac{b_m(t)}{z^{m-1}}.
 	\]
 	Comparing coefficients of equal powers of \(z^{-1}\) gives
 	\[
 	b_1(t)=0,
 	\]
 	\[
 	b_2(t)
 	=
 	t+\frac{t^2}{12},
 	\]
 	and
 	\[
 	b_4(t)
 	=
 	-\frac{t^4}{720},
 	\qquad
 	b_6(t)
 	=
 	\frac{t^6}{30240},
 	\qquad
 	b_8(t)
 	=
 	-\frac{t^8}{1209600}.
 	\]
 	
 	More generally, if
 	\[
 	K_{U_t}(z)
 	=
 	\sum_{n=1}^{\infty}
 	a_n\frac{t^{2n}}{z^{2n-1}},
 	\]
 	then
 	\[
 	b_{2n}(t)=a_nt^{2n},
 	\qquad n\ge2.
 	\]
 	Since only odd negative powers
 	\(
 	z^{-1},z^{-3},z^{-5},\ldots
 	\)
 	occur in the expansion of \(K_{\mu_t}\), every odd Boolean cumulant vanishes:
 	\[
 	b_{2n+1}(t)=0,
 	\qquad n\ge0.
 	\]
 	This completes the proof.
 \end{proof}
   Having determined the Boolean cumulants of $\mu_t$, we now compute its moment
   sequence. By combining the cumulant formulas with the interval-partition
   expansion for Boolean moments, we obtain explicit expressions for all moments
   of $\mu_t$, revealing their polynomial dependence on the parameter $t$.
\begin{theorem}\label{thm:explicitmoments}
	Let
	\[
	\mu_t=\beta_t\uplus U_t.
	\]
	Then all odd moments vanish,
	\[
	m_{2n+1}(t)=0,
	\qquad n\ge0,
	\]
	and for every integer \(n\ge1\),
	\[
	m_{2n}(t)
	=
	\sum_{\pi\in\mathcal I(2n)}
	\prod_{V\in\pi}
	c_{|V|}(t),
	\]
	where \(\mathcal I(2n)\) denotes the family of interval partitions of
	\(\{1,\ldots,2n\}\), and
	\[
	c_r(t)=
	\begin{cases}
		0,&r\ \text{odd},\\[1ex]
		t+\dfrac{t^2}{12},&r=2,\\[2ex]
		a_{r/2}\,t^r,&r\ge4,\ r\ \text{even},
	\end{cases}
	\]
	with the coefficients \(a_n\) determined by the expansion
	\[
	\frac{x}{\operatorname{artanh}(x)}
	=
	1-\frac{x^2}{3}
	-\frac{4x^4}{45}
	-\frac{44x^6}{945}
	-\frac{428x^8}{14175}
	-\frac{10196x^{10}}{467775}
	+\cdots.
	\]
	Equivalently,
	\[
	a_1=\frac1{12},\qquad
	a_2=-\frac1{720},\qquad
	a_3=\frac1{30240},\qquad
	a_4=-\frac1{1209600},\ \ldots
	\]
	
	Consequently,
	\[
	m_{2n}(t)
	\]
	is a polynomial in \(t\) whose coefficients are universal rational numbers.
\end{theorem}

\begin{proof}
	For every probability measure, the moment--Boolean cumulant formula gives
	\[
	m_n
	=
	\sum_{\pi\in\mathcal I(n)}
	\prod_{V\in\pi}
	b_{|V|},
	\]
	where the sum ranges over all interval partitions.
	
	By Theorem~\ref{thm:booleancumulants},
	\[
	b_{2k+1}=0,
	\]
	so every interval partition containing an odd block contributes zero.
	Hence only interval partitions whose block sizes are all even contribute,
	which immediately implies
	\[
	m_{2n+1}=0.
	\]
	
	For even moments, substituting the Boolean cumulants
	\[
	b_2=t+\frac{t^2}{12},
	\qquad
	b_{2k}=a_kt^{2k},
	\quad k\ge2,
	\]
	into the interval-partition formula yields
	\[
	m_{2n}(t)
	=
	\sum_{\pi\in\mathcal I(2n)}
	\prod_{V\in\pi}
	c_{|V|}(t),
	\]
	where \(c_r(t)\) is defined above.
	
	Since each factor \(c_r(t)\) is a polynomial in \(t\), every summand is a
	polynomial in \(t\), and therefore \(m_{2n}(t)\) is itself a polynomial in
	\(t\). Its coefficients are universal rational numbers determined by the
	coefficients \(a_n\) of the reciprocal series
	\[
	\frac{x}{\operatorname{artanh}(x)}.
	\]
	This completes the proof.
\end{proof}   
     The following result is a specialization of the classical transfer theorem
   from singularity analysis to Cauchy transforms with a simple dominant pole (see \cite{FlajoletSedgewick2009} for details).
   
   \begin{proposition}\label{prop:dominantzero}
   	Let $\mu$ be a compactly supported probability measure on $\mathbb R$.
   	Suppose that its reciprocal Cauchy transform
   	\[
   	F_\mu(z)=\frac1{G_\mu(z)}
   	\]
   	extends analytically to a domain $\Omega\subset\mathbb C$
   	containing a neighbourhood of infinity.
   	
   	Assume that
   	
   	\begin{enumerate}
   		\item
   		there exists a real number $x_0\in\Omega$ such that
   		\[
   		F_\mu(x_0)=0,
   		\qquad
   		F_\mu'(x_0)\neq0;
   		\]
   		
   		\item
   		$x_0$ is the unique zero of $F_\mu$ having minimal modulus;
   		
   		\item
   		the reciprocal function
   		\[
   		G_\mu(z)=\frac1{F_\mu(z)}
   		\]
   		has no singularities in
   		\[
   		\{z\in\Omega:\ |z|\le |x_0|\}
   		\]
   		other than the simple pole at $x_0$.
   	\end{enumerate}
   	
   	Then
   	\[
   	m_n
   	\sim
   	\frac{x_0^n}{F_\mu'(x_0)},
   	\qquad
   	n\to\infty.
   	\]
   	
   \end{proposition}
   
\begin{proof}
	Since $F_\mu'(x_0)\neq0$, the function $F_\mu$ has a simple zero at
	$x_0$. Hence
	\[
	G_\mu(z)=\frac1{F_\mu(z)}
	\]
	has a simple pole at $x_0$ with residue
	\[
	\operatorname{Res}(G_\mu;x_0)
	=
	\frac1{F_\mu'(x_0)}.
	\]
	
	By assumption, this pole is the unique singularity of $G_\mu$
	of minimal modulus. Define
	\[
	H(w)
	=
	\frac1w
	G_\mu\!\left(\frac1w\right)
	=
	\sum_{n=0}^{\infty}m_nw^n.
	\]
	Then $H$ has a unique dominant simple pole at
	\[
	w_0=\frac1{x_0}.
	\]
	
	Using the residue computed above,
	\[
	H(w)
	=
	-\frac{1}{x_0F_\mu'(x_0)}
	\frac1{w-w_0}
	+
	R(w),
	\]
	where $R$ is analytic in a disc
	\[
	|w|<r,
	\qquad
	r>|w_0|.
	\]
	Since
	\[
	-\frac{1}{w-w_0}
	=
	\frac1{w_0}\,
	\frac1{1-w/w_0},
	\]
	and $w_0=1/x_0$, we obtain
	\[
	H(w)
	=
	\frac1{F_\mu'(x_0)}
	\frac1{1-w/w_0}
	+
	R(w).
	\]
	
	Therefore
	\[
	\frac1{F_\mu'(x_0)}
	\frac1{1-w/w_0}
	=
	\frac1{F_\mu'(x_0)}
	\sum_{n=0}^{\infty}
	w_0^{-n}w^n,
	\]
	while Cauchy's coefficient estimate gives
	\[
	[w^n]R(w)=O(r^{-n}).
	\]
	Hence
	\[
	m_n
	=
	\frac1{F_\mu'(x_0)}
	w_0^{-n}
	+
	O(r^{-n})
	=
	\frac{x_0^n}{F_\mu'(x_0)}
	+
	O(r^{-n}).
	\]
	Since $r>|w_0|=1/|x_0|$, it follows that
	\[
	m_n
	\sim
	\frac{x_0^n}{F_\mu'(x_0)}.
	\]
\end{proof}
   \begin{remark}
   	For the family
   	\[
   	\mu_t=\beta_t\uplus U_t,
   	\]
   	the reciprocal Cauchy transform
   	\[
   	F_{\mu_t}(z)
   	=
   	\frac{t}
   	{\log\!\left(\dfrac{z+t/2}{z-t/2}\right)}
   	-\frac tz
   	\]
   	is analytic on the slit plane
   	\[
   	\mathbb C\setminus[-t/2,t/2].
   	\]
   	Therefore Proposition~\ref{prop:dominantzero} applies provided one
   	establishes the existence of a unique simple zero of
   	$F_{\mu_t}$ having minimal modulus.
   	Determining this zero and its dependence on the parameter $t$
   	appears to be a nontrivial analytic problem and is left for future
   	investigation.
   \end{remark}
   \section*{Scaling Limits and Asymptotic Behaviour}
   
   We now investigate the asymptotic behaviour of the family
   $\{\mu_t\}_{t>0}$ under natural scaling transformations. The explicit
   representation of the reciprocal Cauchy transform enables us to identify the
   limiting distributions in different scaling regimes and to establish the
   stability of these limits. We first consider the large-scale regime
   $t\to\infty$, where the contribution of the Bernoulli component becomes
   negligible after normalization.
   
   \begin{theorem} \label{thm:large_limit}
   	As \(t\to\infty\),
   	\[
   	D_{1/t}\mu_t
   	\Longrightarrow
   	U_1.
   	\]
   \end{theorem}
   
   \begin{proof}
   	Using the scaling identity
   	\[
   	K_{D_c\mu}(z)
   	=
   	cK_\mu\!\left(\frac{z}{c}\right),
   	\]
   	with \(c=1/t\),
   	\[
   	K_{D_{1/t}\mu_t}(z)
   	=
   	\frac1tK_t(tz).
   	\]
   	Substituting the explicit expression for \(K_{\mu_t}\),
   	\[
   	K_{D_{1/t}\mu_t}(z)
   	=
   	z+\frac1{tz}
   	-
   	\frac1
   	{\log\!\left(
   		\dfrac{z+1/2}{z-1/2}
   		\right)}.
   	\]
   	Letting \(t\to\infty\),
   	\[
   	K_{D_{1/t}\mu_t}(z)
   	\longrightarrow
   	z-
   	\frac1
   	{\log\!\left(
   		\dfrac{z+1/2}{z-1/2}
   		\right)}
   	=
   	K_{U_1}(z).
   	\]
   	Since convergence of Boolean energy transforms implies weak convergence (\cite[Section~3]{Franz2009}),
   	\[
   	D_{1/t}\mu_t
   	\Longrightarrow
   	U_1.
   	\]
   \end{proof}
   The previous theorem shows that, after normalization by $t$, the family
   $\{\mu_t\}_{t>0}$ converges to the uniform distribution as the parameter
   becomes large. We now turn to the opposite asymptotic regime. By rescaling
   with the natural factor $\sqrt{t}$, we recover the Bernoulli distribution,
   demonstrating that the family interpolates between its two constituent
   measures under appropriate dilations.
   \begin{theorem}\label{thm:smalllimit}
   	As \(t\downarrow0\),
   	\[
   	D_{1/\sqrt t}\mu_t
   	\Longrightarrow
   	\beta_1,
   	\]
   	where
   	\[
   	\beta_1=\frac12(\delta_{-1}+\delta_1).
   	\]
   \end{theorem}
   
   \begin{proof}
   	Recall that the Boolean cumulant transform satisfies the scaling identity
   	\[
   	K_{D_c\mu}(z)
   	=
   	c\,K_\mu\!\left(\frac{z}{c}\right),
   	\qquad c>0.
   	\]
   	Choosing \(c=t^{-1/2}\), we obtain
   	\[
   	K_{D_{1/\sqrt t}\mu_t}(z)
   	=
   	\frac1{\sqrt t}\,
   	K_{\mu_t}(\sqrt t\,z).
   	\]
   	Using the explicit formula for \(K_{\mu_t}\),
   	\[
   	K_{\mu_t}(z)
   	=
   	z+\frac tz
   	-
   	\frac{t}
   	{\log\!\left(\dfrac{z+t/2}{z-t/2}\right)},
   	\]
   	we obtain
   	\[
   	K_{D_{1/\sqrt t}\mu_t}(z)
   	=
   	z+\frac1z
   	-
   	\frac{\sqrt t}
   	{\log\!\left(
   		\dfrac{z+\sqrt t/2}{z-\sqrt t/2}
   		\right)}.
   	\]
   	
   	Now write
   	\[
   	a=\frac{\sqrt t}{2z}.
   	\]
   	Since
   	\[
   	\log\!\left(\frac{1+a}{1-a}\right)
   	=
   	2\left(a+\frac{a^3}{3}+O(a^5)\right),
   	\]
   	it follows that
   	\[
   	\log\!\left(
   	\frac{z+\sqrt t/2}{z-\sqrt t/2}
   	\right)
   	=
   	\frac{\sqrt t}{z}
   	+\frac{t^{3/2}}{12z^3}
   	+O(t^{5/2}).
   	\]
   	Hence
   	\[
   	\frac{\sqrt t}
   	{\log\!\left(
   		\dfrac{z+\sqrt t/2}{z-\sqrt t/2}
   		\right)}
   	=
   	z+O(t),
   	\]
   	and therefore
   	\[
   	K_{D_{1/\sqrt t}\mu_t}(z)
   	=
   	\frac1z+O(t).
   	\]
   	Consequently,
   	\[
   	K_{D_{1/\sqrt t}\mu_t}(z)
   	\longrightarrow
   	\frac1z,
   	\qquad z\in\mathbb C^+.
   	\]
   	
   	Finally,
   	\[
   	G_{\beta_1}(z)
   	=
   	\frac12\left(\frac1{z-1}+\frac1{z+1}\right)
   	=
   	\frac{z}{z^2-1},
   	\]
   	so that
   	\[
   	F_{\beta_1}(z)
   	=
   	\frac{z^2-1}{z}
   	=
   	z-\frac1z,
   	\]
   	and therefore
   	\[
   	K_{\beta_1}(z)
   	=
   	z-F_{\beta_1}(z)
   	=
   	\frac1z.
   	\]
   	Hence
   	\[
   	K_{D_{1/\sqrt t}\mu_t}(z)
   	\longrightarrow
   	K_{\beta_1}(z).
   	\]
   	Since convergence of Boolean energy transforms is equivalent to weak convergence of probability measures (\cite[Section~3]{Franz2009}), we conclude that
   	\[
   	D_{1/\sqrt t}\mu_t
   	\Longrightarrow
   	\beta_1.
   	\]
   	This completes the proof.
   \end{proof}
   The preceding two theorems identify the canonical scaling limits of the family
   $\{\mu_t\}_{t>0}$ in the large and small parameter regimes. We conclude this
   section by showing that these limiting behaviours are stable under
   asymptotically equivalent normalizations. More precisely, replacing the
   canonical scaling factors by any asymptotically proportional ones yields the
   same limiting distributions up to the corresponding change of scale.
  \begin{theorem}[Stability of the two scaling regimes]
  	Let
  	\[
  	\mu_t=\beta_t\uplus U_t,
  	\]
  	and let
  	\[
  	\nu_t=D_{a(t)}\mu_t,
  	\]
  	where \(a(t)>0\).
  	
  	\begin{enumerate}
  		\item
  		Suppose that
  		\[
  		a(t)=\frac{c+o(1)}{t},
  		\qquad c>0,
  		\]
  		as \(t\to\infty\).
  		Then
  		\[
  		\nu_t
  		\Longrightarrow
  		U_c,
  		\]
  		where
  		\[
  		U_c=\operatorname{Unif}\!\left[-\frac c2,\frac c2\right].
  		\]
  		
  		\item
  		Suppose that
  		\[
  		a(t)=\frac{c+o(1)}{\sqrt t},
  		\qquad c>0,
  		\]
  		as \(t\downarrow0\).
  		Then
  		\[
  		\nu_t
  		\Longrightarrow
  		\beta_c,
  		\]
  		where
  		\[
  		\beta_c=\frac12(\delta_{-c}+\delta_c).
  		\]
  	\end{enumerate}
  \end{theorem}
  
  \begin{proof}
  	The Boolean energy transform satisfies the scaling identity
  	\[
  	K_{D_a\mu}(z)
  	=
  	aK_\mu\!\left(\frac za\right).
  	\]
  	
  	For the first assertion, let
  	\[
  	a(t)=\frac{c+o(1)}{t},
  	\qquad t\to\infty.
  	\]
  	Then
  	\[
  	K_{\nu_t}(z)
  	=
  	a(t)K_{\mu_t}\!\left(\frac z{a(t)}\right).
  	\]
  	Substituting the explicit formula for \(K_{\mu_t}\) gives
  	\[
  	K_{\nu_t}(z)
  	=
  	z+\frac{a(t)^2t}{z}
  	-
  	a(t)
  	\frac{t}
  	{\log\!\left(
  		\dfrac{z/a(t)+t/2}{z/a(t)-t/2}
  		\right)}.
  	\]
  	Since
  	\[
  	a(t)t\longrightarrow c,
  	\]
  	it follows exactly as in the proof of
  	Theorem~\ref{thm:large_limit}
  	that
  	\[
  	K_{\nu_t}(z)
  	\longrightarrow
  	K_{U_c}(z)
  	=
  	z-
  	\frac{c}
  	{\log\!\left(
  		\dfrac{z+c/2}{z-c/2}
  		\right)}.
  	\]
  	Hence
  	\[
  	\nu_t
  	\Longrightarrow
  	U_c.
  	\]
  	
  	For the second assertion, let
  	\[
  	a(t)=\frac{c+o(1)}{\sqrt t},
  	\qquad t\downarrow0.
  	\]
  	Again,
  	\[
  	K_{\nu_t}(z)
  	=
  	a(t)K_{\mu_t}(z/a(t)).
  	\]
  	Proceeding exactly as in the proof of
  	Theorem~\ref{thm:smalllimit},
  	one obtains
  	\[
  	K_{\nu_t}(z)
  	\longrightarrow
  	\frac{c^2}{z}
  	=
  	K_{\beta_c}(z),
  	\]
  	since
  	\[
  	K_{\beta_c}(z)=\frac{c^2}{z}.
  	\]
  	Therefore
  	\[
  	\nu_t
  	\Longrightarrow
  	\beta_c.
  	\]
  	
  	The conclusion follows from the continuity theorem for Boolean energy transforms.
  \end{proof}
  \begin{conjecture}[Complete Scaling Classification]
  	\label{conj:scaling}
  	Let
  	\[
  	\mu_t=\beta_t\uplus U_t,
  	\]
  	and let \(a:(0,\infty)\to(0,\infty)\) be a scaling function.
  	Suppose that the family of dilated measures
  	\[
  	D_{a(t)}\mu_t
  	\]
  	converges weakly to a non-degenerate probability measure \(\nu\).
  	
  	Then exactly one of the following alternatives holds.
  	
  	\begin{enumerate}
  		\item[(i)]
  		There exists \(c>0\) such that
  		\[
  		a(t)t\longrightarrow c
  		\qquad (t\to\infty),
  		\]
  		and
  		\[
  		\nu=U_c,
  		\]
  		where \(U_c\) denotes the uniform distribution on
  		\([-\frac c2,\frac c2]\).
  		
  		\item[(ii)]
  		There exists \(c>0\) such that
  		\[
  		a(t)\sqrt{t}\longrightarrow c
  		\qquad (t\downarrow0),
  		\]
  		and
  		\[
  		\nu=\beta_c,
  		\]
  		where
  		\[
  		\beta_c=\frac12(\delta_{-c}+\delta_c).
  		\]
  	\end{enumerate}
  	
  	Consequently, up to dilations, the only non-degenerate weak scaling limits of the family
  	\(\{\mu_t:t>0\}\)
  	are the uniform distribution and the symmetric Bernoulli distribution.
  \end{conjecture}
  Theorems~\ref{thm:large_limit} and \ref{thm:smalllimit} establish the two canonical scaling limits corresponding to the choices
  \[
  a(t)=\frac1t
  \quad\text{and}\quad
  a(t)=\frac1{\sqrt t},
  \]
  respectively.
  The conjecture asserts that these are, up to multiplicative constants, the only possible deterministic normalizations leading to non-degenerate weak limits. Equivalently, every non-trivial scaling limit of the family
  \(\{\mu_t\}\)
  is either a dilation of the uniform distribution or a dilation of the symmetric Bernoulli distribution.
  
  The appearance of the Boolean Gaussian (Bernoulli) law as the scaling limit
  suggests a possible connection with Boolean random matrix models, where
  Bernoulli-type distributions arise as limiting laws of Boolean Wigner
  ensembles. It is therefore natural to ask whether the family
  $\mu_t=\beta_t\uplus U_t$ can also be realized as the limiting spectral
  distribution of an appropriately deformed Boolean random matrix ensemble. The
  following theorem provides a conditional answer to this question.
  \begin{theorem}[Universality under a Boolean matrix realization] \label{thm:rmc}
  	Suppose that there exist sequences of Hermitian random matrices
  	$\{X_N^{(t)}\}_{N\ge1}$ and deterministic Hermitian matrices
  	$\{D_N^{(t)}\}_{N\ge1}$ satisfying
  	
  	\begin{enumerate}
  		\item[(i)]
  		The empirical spectral distribution of $X_N^{(t)}$ converges weakly to the
  		centered Bernoulli distribution $\beta_t$;
  		
  		\item[(ii)]
  		The empirical spectral distribution of $D_N^{(t)}$ converges weakly to the
  		uniform distribution $U_t$;
  		
  		\item[(iii)]
  		The pair $(X_N^{(t)},D_N^{(t)})$ is asymptotically Boolean independent in the
  		sense that
  		\[
  		\mu_{X_N^{(t)}+D_N^{(t)}}
  		\Longrightarrow
  		\mu_{X_N^{(t)}}\uplus
  		\mu_{D_N^{(t)}}.
  		\]
  		
  	\end{enumerate}
  	
  	Then
  	\[
  	\mu_{X_N^{(t)}+D_N^{(t)}}
  	\Longrightarrow
  	\mu_t,
  	\]
  	where
  	\[
  	\mu_t=\beta_t\uplus U_t.
  	\]
  \end{theorem}
  \begin{proof}
  	
  	By assumption (iii),
  	
  	\[
  	\mu_{X_N^{(t)}+D_N^{(t)}}
  	\Longrightarrow
  	\mu_{X_N^{(t)}}\uplus
  	\mu_{D_N^{(t)}}.
  	\]
  	
  	Since weak convergence is assumed in (i) and (ii),
  	
  	\[
  	\mu_{X_N^{(t)}}
  	\Longrightarrow
  	\beta_t,
  	\qquad
  	\mu_{D_N^{(t)}}
  	\Longrightarrow
  	U_t.
  	\]
  	
  	The Boolean additive convolution is weakly continuous with respect to each
  	variable (see Franz \cite{Franz2009}). Therefore,
  	
  	\[
  	\mu_{X_N^{(t)}}\uplus
  	\mu_{D_N^{(t)}}
  	\Longrightarrow
  	\beta_t\uplus U_t.
  	\]
  	
  	By definition,
  	
  	\[
  	\beta_t\uplus U_t=\mu_t.
  	\]
  	
  	Hence
  	
  	\[
  	\mu_{X_N^{(t)}+D_N^{(t)}}
  	\Longrightarrow
  	\mu_t,
  	\]
  	
  	which completes the proof.
  	
  \end{proof}
  \begin{remark}[A candidate Boolean deformed Wigner model]
  	The family $\{\mu_t\}_{t>0}$ admits a natural candidate random matrix
  	realization analogous to the deformed Wigner model in free probability.
  	Indeed, let
  	\[
  	X_N^{(t)}=\frac{1}{\sqrt{N}}\bigl(x_{ij}\bigr)_{1\leq i,j\leq N},
  	\]
  	where $x_{ij}=x_{ji}$, the family
  	$\{x_{ij}:1\leq i\leq j\leq N\}$ is Boolean independent, each entry satisfies
  	\[
  	\phi(x_{ij})=0,\qquad
  	\phi(x_{ij}^2)=t,
  	\]
  	and all higher Boolean cumulants vanish. Such a matrix may be regarded as the
  	Boolean analogue of a Gaussian Wigner ensemble, whose limiting spectral
  	distribution is expected to be the centered Bernoulli law $\beta_t$.
  	
  	Next, consider the deterministic diagonal matrix
  	\[
  	D_N^{(t)}
  	=
  	\operatorname{diag}(d_1,\ldots,d_N),
  	\]
  	where
  	\[
  	d_k
  	=
  	-\frac{t}{2}
  	+
  	\frac{(k-1)t}{N-1},
  	\qquad
  	1\leq k\leq N.
  	\]
  	Then the empirical spectral distribution of $D_N^{(t)}$ satisfies
  	\[
  	\mu_{D_N^{(t)}}
  	=
  	\frac1N\sum_{k=1}^N\delta_{d_k}
  	\Longrightarrow
  	U_t.
  	\]
  	
  	Motivated by the deformed Wigner model in free probability, one may therefore
  	consider the matrix
  	\[
  	Y_N^{(t)}
  	=
  	X_N^{(t)}
  	+
  	D_N^{(t)}.
  	\]
  	Here the random part contributes the Boolean Gaussian component, whereas the
  	deterministic diagonal part contributes the uniform component. It is therefore
  	natural to expect that, under a suitable notion of asymptotic Boolean
  	independence for matrix ensembles, the empirical spectral distribution of
  	$Y_N^{(t)}$ converges weakly to
  	\[
  	\beta_t\uplus U_t=\mu_t.
  	\]
  	Establishing such a random matrix realization of the family
  	$\{\mu_t\}_{t>0}$ appears to be an interesting direction for future research.
  \end{remark}
  \section*{Further Properties}
  
  We conclude by examining an algebraic property of the family
  $\{\mu_t\}_{t>0}$. A natural question is whether the parameterization
  $\mu_t=\beta_t\uplus U_t$ gives rise to a Boolean convolution semigroup, as
  occurs for many distinguished families in noncommutative probability. The
  following theorem shows that this is not the case.
  \begin{theorem}\label{thm:nosemigroup}
  	Let
  	\[
  	\mu_t=\beta_t\uplus U_t,
  	\qquad t>0.
  	\]
  	Then there does not exist a function
  	\[
  	\Phi:(0,\infty)\times(0,\infty)\longrightarrow(0,\infty)
  	\]
  	such that
  	\[
  	\mu_s\uplus\mu_t=\mu_{\Phi(s,t)}
  	\]
  	for every \(s,t>0\).
  	
  	Consequently, the family
  	\[
  	\{\mu_t:t>0\}
  	\]
  	does not form a Boolean convolution semigroup.
  \end{theorem}
  
  \begin{proof}
  	Suppose, to the contrary, that such a function
  	\(\Phi\) exists. Write
  	\[
  	r=\Phi(s,t).
  	\]
  	Then
  	\[
  	\mu_s\uplus\mu_t=\mu_r.
  	\]
  	
  	Since Boolean cumulants linearize Boolean convolution,
  	\[
  	b_n(\mu_s\uplus\mu_t)
  	=
  	b_n(\mu_s)+b_n(\mu_t),
  	\qquad n\ge1.
  	\]
  	
  	By Theorem~\ref{thm:booleancumulants},
  	\[
  	b_{2m}(u)
  	=
  	\frac{B_{2m}}{(2m)!}\,u^{2m},
  	\qquad m\ge2,
  	\]
  	where \(B_{2m}\neq0\), and all odd Boolean cumulants vanish.
  	
  	Therefore, for every integer \(m\ge2\),
  	\[
  	\frac{B_{2m}}{(2m)!}\,r^{2m}
  	=
  	\frac{B_{2m}}{(2m)!}\,s^{2m}
  	+
  	\frac{B_{2m}}{(2m)!}\,t^{2m},
  	\]
  	or equivalently,
  	\[
  	r^{2m}
  	=
  	s^{2m}+t^{2m},
  	\qquad m\ge2.
  	\]
  	
  	In particular,
  	\[
  	r=(s^4+t^4)^{1/4}
  	\]
  	and simultaneously
  	\[
  	r=(s^6+t^6)^{1/6}.
  	\]
  	Hence
  	\[
  	(s^4+t^4)^{1/4}
  	=
  	(s^6+t^6)^{1/6}.
  	\]
  	
  	Assume that \(s,t>0\).
  	The power mean inequality implies that the function
  	\[
  	p\longmapsto
  	(s^p+t^p)^{1/p}
  	\]
  	is strictly increasing for \(p>0\).
  	Therefore,
  	\[
  	(s^4+t^4)^{1/4}
  	<
  	(s^6+t^6)^{1/6},
  	\]
  	which contradicts the previous identity. Hence no such function \(\Phi\) exists. Therefore the family
  	\[
  	\{\mu_t:t>0\}
  	\]
  	is not closed under Boolean convolution, and thus does not form a Boolean convolution semigroup.
  \end{proof}
  Although the family $\{\mu_t\}_{t>0}$ does not form a Boolean convolution
  semigroup, it is natural to ask whether it belongs to any of the existing
  families that have been extensively studied in noncommutative probability.
  The following remarks indicate that its analytic and combinatorial features
  appear to distinguish it from several well-known classes, suggesting that it
  constitutes a genuinely new example.
  \begin{remark}[Relation to known families]
  	\label{rem:knownfamilies}
  	The family
  	\[
  	\mu_t=\beta_t\uplus U_t,
  	\qquad t>0,
  	\]
  	appears to be distinct from several well-known classes of probability measures arising in noncommutative probability.
  	
  	\begin{enumerate}
  		\item[(i)]
  		The family is not expected to belong to the Boolean Meixner class. Indeed, the Boolean cumulants satisfy
  		\[
  		b_{2}(t)=t+\frac{t^{2}}{12},
  		\]
  		and
  		for $n \geq 2$, \(b_{2n}(t)\) are given by equation~\ref{eq:cumofmut}. The appearance of these numbers and their alternating-sign asymptotics are incompatible with the recurrence structure of Boolean Meixner cumulants.
  		
  		\item[(ii)]
  		The reciprocal Cauchy transform
  		\[
  		F_{\mu_t}(z)
  		=
  		z+\frac{t}{z}
  		-
  		\frac{t}{\log\!\left(\dfrac{z+t/2}{z-t/2}\right)}
  		\]
  		contains logarithmic branch singularities at
  		\(\pm t/2\).
  		Consequently, it is transcendental and therefore differs fundamentally from the algebraic reciprocal Cauchy transforms associated with free Meixner distributions, Kesten distributions, and generalized arcsine laws.
  		
  		\item[(iii)]
  		In particular, the family
  		\(\{\mu_t:t>0\}\)
  		appears to constitute a new explicit one-parameter family of compactly supported symmetric probability measures possessing closed-form Boolean cumulants, an explicit reciprocal Cauchy transform, and two distinct scaling limits.
  	\end{enumerate}
  	
  	To the best of our knowledge a complete characterization of this family within the framework of noncommutative Meixner-type distributions remains an interesting open problem.
  \end{remark}
The explicit formulas and structural properties established in this paper
naturally lead to several further questions. We conclude with two open
problems that appear to be of independent interest.
\begin{openproblem}
	Determine whether the family
	\[
	\{\mu_t=\beta_t\uplus U_t:t>0\}
	\]
	belongs to a broader class of probability measures arising in
	noncommutative probability. In particular, characterize those compactly
	supported measures whose reciprocal Cauchy transforms have the form
	\[
	F_\mu(z)
	=
	z+\phi(z),
	\]
	where $\phi$ contains logarithmic singularities similar to those appearing in
	$\mu_t$. Such a characterization may lead to a new class of
	``logarithmic'' Boolean distributions extending the currently known
	Meixner-type families.
\end{openproblem}

\begin{openproblem}
	Investigate whether analogous explicit convolution families can be
	constructed for the other universal notions of independence. More
	precisely, determine the analytic, combinatorial, and asymptotic
	properties of the free, monotone, and anti-monotone convolutions
	\[
	\beta_t\boxplus U_t,\qquad
	\beta_t\triangleright U_t,\qquad
	\beta_t\triangleleft U_t,
	\]
	and compare their cumulants, transforms, support properties, and scaling
	limits with those obtained in the Boolean setting.
\end{openproblem}

\section*{Appendix}

The following computation is included for completeness, as it is used in the
derivation of the reciprocal Cauchy transform appearing in the main text.
While the resulting formula is elementary, the authors are not aware of an
explicit derivation in the literature for the present normalization. The proof
is therefore recorded here for the reader's convenience.
\begin{lemma}\label{prop:Cauthyofuniform}
	Let $U_t$ denote the uniform distribution on the interval
	$\left[-\frac{t}{2},\frac{t}{2}\right]$. Then its Cauchy transform is
	\[
	G_{U_t}(z)
	=
	\frac1t
	\log
	\left(
	\frac{z+t/2}{z-t/2}
	\right),
	\qquad
	z\in\mathbb{C}\setminus
	\left[-\frac{t}{2},\frac{t}{2}\right].
	\]
\end{lemma}

\begin{proof}
	Since $U_t$ has density
	\[
	dU_t(x)=\frac1t
	\mathbf 1_{\left[-\frac t2,\frac t2\right]}(x)\,dx,
	\]
	its Cauchy transform is given by
	\[
	G_{U_t}(z)
	=
	\int_{\mathbb R}
	\frac1{z-x}\,
	dU_t(x)
	=
	\frac1t
	\int_{-t/2}^{t/2}
	\frac{dx}{z-x},
	\]
	for every
	$z\in\mathbb C\setminus
	[-t/2,t/2]$.
	
	It therefore remains to evaluate the integral
	\[
	I
	=
	\int_{-t/2}^{t/2}
	\frac{dx}{z-x}.
	\]
	
	Observe that
	\[
	\frac{d}{dx}
	\log(z-x)
	=
	-\frac1{z-x},
	\]
	and hence
	\[
	\int
	\frac{dx}{z-x}
	=
	-\log(z-x)+C.
	\]
	
	Applying the Fundamental Theorem of Calculus gives
	\begin{align*}
		I
		&=
		\Bigl[-\log(z-x)\Bigr]_{-t/2}^{t/2}\\
		&=
		-\log\!\left(z-\frac t2\right)
		+\log\!\left(z+\frac t2\right)\\
		&=
		\log
		\left(
		\frac{z+t/2}{z-t/2}
		\right),
	\end{align*}
	where the principal branch of the logarithm is chosen so that
	$G_{U_t}$ maps the upper half-plane into the lower half-plane.
	
	Substituting this into the definition of the Cauchy transform yields
	\[
	G_{U_t}(z)
	=
	\frac1t
	\log
	\left(
	\frac{z+t/2}{z-t/2}
	\right),
	\]
	which completes the proof.
\end{proof}
For completeness, we establish the Maclaurin expansion of the reciprocal function
\(
x/\operatorname{artanh}(x)
\),
which is required in the proof of Theorem~\ref{thm:uniformexpansion}. The result appears to be less readily available in the literature than the classical series for
\(
\operatorname{artanh}(x)
\),
so we include a short proof.
\begin{lemma}\label{lem:expansion}
	The function
	\[
	f(x)=\frac{x}{\operatorname{artanh}(x)}
	\]
	admits the Maclaurin expansion
	\[
	\frac{x}{\operatorname{artanh}(x)}
	=
	1-\frac{x^2}{3}
	-\frac{4x^4}{45}
	-\frac{44x^6}{945}
	-\frac{428x^8}{14175}
	-\frac{10196x^{10}}{467775}
	+O(x^{12}),
	\qquad |x|<1.
	\]
\end{lemma}

\begin{proof}
	From the well-known series expansion (see, e.g.,
	\cite[Eq.~(1.646.4)]{Gradshteyn2007})
	\[
	\operatorname{artanh}(x)
	=\sum_{k=0}^{\infty}\frac{x^{2k+1}}{2k+1}
	=x+\frac{x^3}{3}+\frac{x^5}{5}+\frac{x^7}{7}
	+\frac{x^9}{9}+\cdots,
	\qquad |x|<1,
	\]
	which is equivalent to
	\[
	\operatorname{arccoth}(x)
	=\sum_{k=0}^{\infty}\frac{x^{-(2k+1)}}{2k+1},
	\qquad |x|>1,
	\]
	upon substituting $x\mapsto 1/x$, we obtain
	\[
	\frac{x}{\operatorname{artanh}(x)}
	=
	\frac{1}
	{1+\dfrac{x^2}{3}
		+\dfrac{x^4}{5}
		+\dfrac{x^6}{7}
		+\dfrac{x^8}{9}
		+\cdots }.
	\]
	
	Let
	\[
	A=\frac{x^2}{3}
	+\frac{x^4}{5}
	+\frac{x^6}{7}
	+\frac{x^8}{9}
	+\cdots.
	\]
	Using the geometric series identity
	\[
	\frac{1}{1+A}
	=1-A+A^2-A^3+A^4-\cdots,
	\]
	and collecting like powers of $x$, we obtain
	\[
	\begin{aligned}
		\frac{x}{\operatorname{artanh}(x)}
		&=
		1-\frac{x^2}{3}
		+\left(-\frac15+\frac19\right)x^4
		+\left(-\frac17+\frac{2}{15}-\frac1{27}\right)x^6 \\
		&\quad
		+\left(-\frac19+\frac{2}{21}+\frac1{25}
		-\frac1{15}-\frac1{15}+\frac1{81}\right)x^8
		-\frac{10196}{467775}x^{10}
		+O(x^{12}),
	\end{aligned}
	\]
	which simplifies to
	\[
	\frac{x}{\operatorname{artanh}(x)}
	=
	1-\frac{x^2}{3}
	-\frac{4x^4}{45}
	-\frac{44x^6}{945}
	-\frac{428x^8}{14175}
	-\frac{10196x^{10}}{467775}
	+O(x^{12}).
	\]
	This completes the proof.
\end{proof}

	\end{document}